\documentclass[10pt,fleqn]{article}

\usepackage{graphicx}%
\usepackage{multirow}%
\usepackage{amsmath,amssymb,amsfonts}%
\usepackage{amsthm}%
\usepackage{mathrsfs}%
\usepackage[title]{appendix}%
\usepackage{xcolor}%
\usepackage{textcomp}%
\usepackage{manyfoot}%
\usepackage{booktabs}%
\usepackage{algorithm}%
\usepackage{algorithmicx}%
\usepackage{algpseudocode}%
\usepackage{listings}%
\theoremstyle{thmstyleone}%
\newtheorem{theorem}{Theorem}%  meant for continuous numbers
\newtheorem{proposition}[theorem]{Proposition}% 

\theoremstyle{thmstyletwo}%
\newtheorem{example}{Example}%
\newtheorem{remark}{Remark}%

\theoremstyle{thmstylethree}%
\newtheorem{definition}{Definition}%

\newtheorem{assumption}[theorem]{Assumption}

\newcommand{\mR}{\mathbb{R}}
\newcommand{\mC}{\mathbb{C}}
\newcommand{\X}{\mathcal{X}}

\newcommand{\D}{\mathcal{D}}
\newcommand{\K}{\mathcal{K}}
\newcommand{\C}{\mathcal{C}}
\newcommand{\cO}{\mathcal{O}}

\newcommand{\V}{\mathcal{V}}

\newcommand{\Y}{\mathcal{Y}}
\newcommand{\U}{\mathcal{U}}
\newcommand{\E}{\mathcal{E}}
\newcommand{\F}{\mathcal{F}}

\newcommand{\cL}{\mathcal{L}}
\newcommand{\cH}{\mathcal{H}}

\newcommand{\pperp}{\perp \!\!\!\perp}
\newcommand{\bq}{\begin{equation}}
\newcommand{\eq}{\end{equation}}
\newcommand{\bma}{\begin{bmatrix}}
\newcommand{\ema}{\end{bmatrix}}

\begin{document}

\title{Symmetry in linear physical systems}

\author{Arjan van der Schaft, Rodolphe Sepulchre, Thomas Chaffey
\thanks{A.J. van der Schaft is with the Bernoulli Institute for Mathematics, Computer
Science and Artificial Intelligence, Jan C. Willems Center for Systems and Control, University of Groningen, PO Box 407, 9700 AK, the
Netherlands, {\tt\small A.J.van.der.Schaft@rug.nl}, R. Sepulchre is with the Dept. of Electrical Engineering (STADIUS), KU Leuven, Belgium, and Dept. of Engineering (Control Group), University of Cambridge, UK, {\tt\small rodolphe.sepulchre@kuleuven.be}, T. Chaffey is with School of Electrical and Computer Engineering, University of Sydney, Australia, {\tt\small thomas.chaffey@sydney.edu.au}.}
}

\maketitle

\abstract{Physical systems with symmetry arise abundantly in applications, and are endowed with interesting mathematical structures. The present paper focusses on linear reciprocal and input-output Hamiltonian systems. Their characterization is studied from an input-output as well as from a state point of view. Geometrically, it turns out that they both define Lagrangian subspaces with corresponding generating functionals. Furthermore, the relations with time reversibility are analyzed. The system classes under consideration are expected to admit scalable control laws, and to be important building blocks in design.}

\section{Introduction}\label{sec1}

In this paper we will be concerned with linear physical system classes that are characterized by some form of \emph{symmetry}. The interest in such system classes has a long history, often motivated by electrical and mechanical synthesis problems such as those occurring in the design of analog computers. Recently, there is renewed interest due to various reasons. In large scale engineering systems there is a clear need to exploit extra structure for analysis and control purposes. This may lead to improved scalability and to controller structures which are simpler, more robust, and have a clear physical interpretation; see e.g. \cite{pates, SIAMgraphs, pass}. Furthermore, the emerging area of \emph{neuromorphic computation} strongly motivates the identification of suitable classes of system components to be used for the design of neuro-computing devices, neuro-sensors, and neuro-controllers. 

While classical electrical network synthesis was for a large part concerned with \emph{linear} system components, for many applications this emphasis on linear systems does not suffice. On the other hand, synthesis with \emph{arbitrary} nonlinear system components seems problematic. Hence an insightful theory of linear system components with well-defined symmetry structures could pave the way to the identification of proper nonlinear generalizations as well. Therefore, in the present paper we will concentrate on \emph{linear} systems, and leave their nonlinear extensions for future work. For some already existing nonlinear explorations we refer to e.g. \cite{sepulchre24,SIAMreciprocal} and the references quoted therein.

\smallskip

{\bf Notation} Throughout this paper we consider standard finite-dimen-sional linear time-invariant (LTI) systems $\Sigma=(A,B,C,D)$ given in the ubiquitous input-state-output form
\bq
\label{system}
\Sigma: \quad 
\begin{array}{rcl}
\dot{x} & = & Ax + Bu,\qquad x \in \X=\mR^n, \; u \in \U=\mR^m
\\[2mm]
y & = & Cx +Du, \qquad y \in \Y=\mR^m
\end{array}
\eq
with equal number of inputs and outputs.
The \emph{impulse response matrix} of $\Sigma$ is denoted by
\bq
W(t-\tau) := Ce^{A(t-\tau)}B + D \delta (t-\tau),
\eq
and defines the \emph{Volterra integral operator}
\bq
y(t) = \int_{-\infty}^t W(t- \tau)u(\tau) d \tau + Du(t), \quad t \in \mR,
\eq
from input functions $u$ to output functions $y$ (with $\Sigma$ initially at rest, that is, $x(-\infty)=0$). Furthermore, the \emph{transfer matrix} of $\Sigma$ is given by
\bq
K(s):= C(Is-A)^{-1}B +D.
\eq
We primarily consider two classes of systems with symmetry: namely \emph{reciprocal systems} (including relaxation systems) in Section 2, and \emph{input-output Hamiltonian systems} in Section 3. These two classes can be expected to be fundamental building blocks in any network theory of physical and computing systems. Section 4 brings in another form of symmetry, namely \emph{time reversibility}, and its relations to reciprocal and input-output Hamiltonian systems. Section 5 contains conclusions and further outlook. Required background on Lagrangian subspaces and Dirac structures is given in Appendix \ref{secA1}. Appendix \ref{secA2} contains further explorations on the kernel of the Hankel operator of a reciprocal system.
%In the last section we also analyze the possible combination of two properties that seemingly exclude each other; namely \emph{reciprocity} and \emph{losslessness}.

\section{Reciprocal systems}
\subsection{Basic definitions and properties}
First systems under consideration are \emph{reciprocal systems}, including the subclass of \emph{relaxation systems}.

Recall that a \emph{signature matrix} $\sigma$ is a diagonal matrix with elements $+1$ and $-1$ on the diagonal. Hence $\sigma^2=I$, and thus the linear mapping corresponding to $\sigma$ is an \emph{involution}. Conversely, any involution $\sigma: \F \to \F$ for some finite-dimensional linear space $\F$ can be transformed into a signature matrix by a similarity transformation.
\begin{definition}
A system $\Sigma$ is \emph{reciprocal} with respect to a signature matrix $\sigma$ if its transfer matrix $K(s)$ satisfies $\sigma K(s)=K^\top (s) \sigma$, or equivalently its impulse response matrix $W(t-\tau)$ satisfies $\sigma W(t-\tau)= W^\top(t-\tau) \sigma$.
\end{definition}
\begin{remark}
{\rm
Reciprocity amounts to the following symmetry of the input-output behavior of $\Sigma$. For simplicity, let $\sigma=I_m$. Consider the system initialized at $x(-\infty)=0$. Apply an input function on $(-\infty,\infty)$ with all components equal to zero, except for the $i$-the component $u_i(\cdot)=f(\cdot)$, and observe the $j$-th output component $y_j(\cdot)$. Next apply an input function with all components equal to zero, except for the $j$-the component $u_j(\cdot)$, where $u_j(\cdot)=f(\cdot)$ is the same scalar function as before. Then the observed $i$-th output component $y_i(\cdot)$ is \emph{equal} to $y_j(\cdot)$, and this symmetry property should hold for all $i,j=1, \cdots,m$. An equivalent formulation is given in classical electrical network theory \cite{newcomb} as follows. Take any two input-output function pairs $(u_a,y_a), (u_b,y_b)$ of $\Sigma$. Then $\Sigma$ is reciprocal if and only if
\bq
u_a \star y_b = u_b \star y_a,
\eq
where $\star$ denotes convolution on $(-\infty,\infty)$.
}
\end{remark}
In \cite{willems72} it was shown that a system $\Sigma$ with minimal state space realization $(A,B,C,D)$ is reciprocal with respect to the signature matrix $\sigma$ if and only if there exists an invertible matrix $G=G^\top$ (which is unique) such that
%\bq
%\label{G1}
%\bma G & 0 \\[2mm] 0 & \sigma \ema \bma A & B \\[2mm] C & D \ema =
%\bma A & B \\[2mm] C & D \ema^\top \bma G & 0 \\[2mm] 0 & \sigma \ema ,
%\eq
%or written out
\bq
\label{G12}
A^\top G + GA = 0, \; B^\top G= \sigma C, \; \sigma D= D^\top \sigma.
\eq
The matrix $G$ defines a, possibly indefinite, inner product on the state space $\X$; in the rest of the paper referred to as a \emph{pseudo-inner product}. Reciprocal systems arise abundantly in applications \cite{anderson,willems72}. For example, in electrical network theory reciprocal systems that are also \emph{passive} (see Section \ref{subsec:rec}) correspond to electrical networks containing capacitors, inductors, resistors, as well transformers (but \emph{no} gyrators); cf. \cite{anderson}, \cite{willems72}, \cite{hughes}.

By defining the symmetric matrix $P=-GA$, any reciprocal system can be written as a \emph{pseudo-gradient system}
\bq
\label{gradient}
\begin{array}{rcl}
G \dot{x} & = & -Px + C^\top \sigma u
\\[2mm]
y & = & Cx +Du, \qquad \sigma D= D^\top \sigma,
\end{array}
\eq
with potential function $\frac{1}{2}x^\top P x$; see \cite{SIAMreciprocal} and the references quoted therein.
Equivalently, the system $\Sigma$ with inputs $u$ and outputs $\sigma y$ is reciprocal whenever it has the same input-output behavior as its \emph{dual system} 
%$(A^\top,C^\top,B^\top,D^\top)$, 
defined as
\bq
\label{dual}
\Sigma^d: \quad \begin{array}{rcl}
\dot{z} & = & A^\top z + C^\top  u^d
\\[2mm]
y^d & = & B^\top z +  D^\top  u^d
\end{array}
\eq
with inputs $u^d=\sigma u$ and outputs $y^d$. Indeed, the unique state space isomorphism between $\Sigma$ and $\Sigma^d$ is given as $z=Gx$. 
\begin{remark}
{\rm Note furthermore that the time-reversed dual state $\hat{z}(t)=z(-t)$ satisfies
\bq
\dot{\hat{z}}(t) = - A^\top \hat{z}(t) - C^\top \hat{u}^d(t),
\eq
with $\hat{u}^d(t)=u^d(-t)$. These are the state space equations of the \emph{adjoint system} to be defined later on in \eqref{adjoint}.}
\end{remark}

One verifies, \cite{willems72,SIAMreciprocal}, that along the solutions of any reciprocal system
\bq
\frac{d}{dt} x^\top (-t)Gx(t) = - u^\top (-t) \sigma y(t) + \sigma y^\top(-t) u(t),
\eq
Hence, by considering solutions with $u(t)=0, t\geq 0$, and integrating over $[0,\infty)$, one obtains
$x^\top (-\infty)Gx(\infty) - x^\top(0)Gx(0) = - \int_0^\infty u^\top(-t) \sigma y(t) dt$. Thus if either $x(-\infty)$ or $x(\infty)$ equals zero, then $G$ is determined by
\bq
\label{G2}
x^\top(0)Gx(0) = \int_0^\infty u^\top (-t) \sigma y(t) dt
\eq
for any $x(0) \in \X$. In particular, if the system is \emph{controllable} then any $x(0)$ can be reached from $x(-\infty)=0$, implying that $G$ is fully determined by the input-output behavior of the system.
%\begin{proposition}
%If the reciprocal system \eqref{system} is controllable, then $G$ is uniquely determined by its input-output behavior.
%\end{proposition}

Using \eqref{G12} the impulse response matrix $W(t-\tau)= Ce^{A(t-\tau)}B + D \delta (t-\tau)$ of a reciprocal system $\Sigma$ can be rewritten as
\bq
\label{W}
\begin{array}{rcl}
\sigma W(t-\tau) &= &B^\top G e^{At} e^{-A\tau}B + \sigma D \delta (t-\tau) \\[2mm]
&= &B^\top e^{A^\top t} G e^{-A\tau}B + D^\top \sigma \delta (t-\tau).
\end{array}
\eq
This motivates to consider reciprocal systems from the \emph{Hankel} point of view as follows.
For any past input function $u_p: (-\infty,0] \to \mR^m$ define the \emph{time-reversed} input $\hat{u}_p: [0,\infty) \to \mR^m$ as $\hat{u}(t):=u_p(-t), t \in [0,\infty)$.
Then the Hankel operator $\cH : L_2 ([0,\infty),\mR^m) \to L_2 ([0,\infty),\mR^m)$
%, from time-flipped input functions $u_f$ on $[0,\infty)$ to the output functions $y$ on $[0,\infty)$, 
is given as the map $\hat{u}_p \mapsto y_f$ given as
\bq
y_f(t) =  \sigma B^\top e^{A^\top t} G \int_0^\infty e^{A\tau}B\hat{u}_p(\tau) d\tau, \; t \in [0,\infty),
\eq
where $y_f(t), t \in [0,\infty)$, is the future output of $\Sigma$ resulting from $x(-\infty)=0$ and the $L_2$ input $u(t)=u_p(t), t \in (-\infty,0], u(t)=0, t \in [0,\infty)$. 
In order that this map is well-defined (i.e., $y_f(\cdot) \in L_2([0,\infty), \mR^m)$) we impose in the rest of this section the following assumption.
\begin{assumption}
$A$ is Hurwitz.
\end{assumption}
This implies $x(\infty)=0$, and hence by \eqref{G2} the matrix $G$ is uniquely determined by the input-output behavior (even if the system $\Sigma$ is not minimal). For simplicity of notation $L_2([0,\infty), \mR^m)$ will be throughout abbreviated to $L_2[0,\infty)$.

It follows from \eqref{W} that the kernel of $\sigma \cH$ is given by
\bq
\label{symkernel}
B^\top e^{A^\top t} G e^{A \tau} B,
\eq
and thus is \emph{symmetric}; see Appendix \ref{secA2} for its orthonormal decomposition. Hence $\cH$ satisfies $\sigma \cH= \cH^* \sigma,$
where ${}^*$ denotes \emph{adjoint} of an operator. Said otherwise, the operator $\sigma \cH$ is \emph{self-adjoint}, and thus there is an associated \emph{quadratic functional} $\mathfrak{H}: L_2[0,\infty) \to \mR$, defined as
\bq
\label{quad}
\begin{array}{l}
\mathfrak{H}(\hat{u}_p)=\frac{1}{2} \! < \hat{u}_p , \sigma \cH (\hat{u}_p) > = \\[2mm] \frac{1}{2} \! \left(\int_0^\infty e^{At}Bu_p(-t) dt \right)^\top \! G \int_0^\infty e^{A\tau}Bu_p(-\tau) d\tau.
\end{array}
\eq
Indeed, see \cite{chaffey23} for details, the \emph{variational derivative} of $\mathfrak{H}$ equals $\sigma \cH$, i.e.,
\bq
\frac{\delta \mathfrak{H}}{\delta \hat{u}_p}(\hat{u}_p) = \sigma \cH(\hat{u}_p).
\eq
Note as well that $\mathfrak{H}(\hat{u}_p)= \frac{1}{2}x^\top G x$ with $x=x(0)$, and thus the quadratic functional $\mathfrak{H}$ factorizes over the state space $\X$. In this sense, $\mathfrak{H}$ is a \emph{memory functional}. 
%In fact, as discussed above, $\mathfrak{H}(\hat{u}_p)= \frac{1}{2}x^\top G x$ defines a storage function
Furthermore, instead of splitting $(-\infty,\infty)$ into the past $(-\infty,0)$ and the future $[0,\infty)$ at present time $0$, by time-invariance one can also split $(-\infty,\infty)=(-\infty,t) \cup [t,\infty)$ for any time $t\geq 0$. In this way $\mathfrak{H}$ becomes equal to $\frac{1}{2}x(t)^\top G x(t)$; see \cite{chaffey23}.

\subsection{Geometric formulation}
Reciprocity can be interpreted from a coordinate-free and geometric point of view as follows.
Given a signature matrix $\sigma$ consider on the product space $L_2[0,\infty) \times L_2[0,\infty)$ with elements $(f,e)$ 
%of time-reversed past input functions $\hat{u}_p$ and signature matrix modified future output functions $\sigma y_f$ 
the weak \emph{symplectic form} \cite{abraham}
%\bq
%\label{weaksym}
%\ll (\hat{u}_p^1, \sigma y_f^1), (\hat{u}_p^2,  \sigma y_f^2) \gg:= \int_0^\infty u_f^2(t)^\top (\sigma y^1(t)) - u_f^1(t)^\top (\sigma y^2(t))   dt.
%\eq
\bq
\label{weaksym}
\ll (f^a, e^a), (f^b, e^b) \gg:= \int_0^\infty f^b(t)^\top e^a(t) - f^a(t)^\top e^b(t)   dt.
\eq
A subspace $\cL \subset L_2(0,\infty) \times L_2(0,\infty)$ is \emph{Lagrangian} if and only if $\cL= \cL^{\pperp}$, where ${}^{\pperp}$ means orthogonal companion with respect to the symplectic form $\ll \cdot,\cdot \gg$; see Appendix \ref{secA1}. A \emph{generating functional} for a Lagrangian subspace $\cL$ is a functional $\mathfrak{V}: L_2[0,\infty) \to \mR$ such that
\bq
\label{genfun}
\cL = \mbox{ graph } \frac{\delta \mathfrak{V}}{\delta f}(f),
\eq
where $\frac{\delta \mathfrak{V}}{\delta f}$ is the variational derivative of $\mathfrak{V}$. Any subspace $\cL$ of the form \eqref{genfun} is a Lagrangian subspace, and furthermore for any Lagrangian subspace $\cL$ that can be parametrized by $f$ there exists a generating functional $\mathfrak{V}$. More generally, let $\cL$ be a Lagrangian subspace and consider a splitting $f=(f^1,f^2), e=(e^1,e^2)$ such that $\cL$ is parametrized by $u:=(f^1,e^2)$ (such a splitting always exists). Then
\bq
\cL = \{ (f^1,f^2,e^1,e^2) \mid y= \sigma \frac{\delta \mathfrak{V}}{\delta u}(u), \mbox{with } u=(f^1,e^2), y:=(e^1,f^2)\}
\eq
with $\sigma$ is the signature matrix corresponding to the splitting, i.e. $\dim f^1$ is the number of elements $+1$ and $\dim f^2$ is the number of elements $-1$ on the diagonal of $\sigma$.
The following proposition is immediate.
\begin{proposition}
A system $\Sigma$ is reciprocal with respect to $\sigma$ if and only if
\bq
\begin{array}{l}
\cL := \{(\hat{u}_p, y_f):[0,\infty) \to \mR^m \times \mR^m \mid y_f(t), t \in [0,\infty), \mbox{ resulting}\\[2mm]
  \mbox{from } u(t)=u_p(t), t \in (-\infty,0], u(t)=0, t\in [0,\infty), x(-\infty)=0\}
\end{array}
\eq
is a Lagrangian subspace. Furthermore, the generating functional of $\cL$ is the functional $\mathfrak{H}$ defined in \eqref{quad}.
\end{proposition}
%\begin{proof}
%$\cL = \cL^{\pperp}$
%\end{proof}
%The symplectic form $\ll \cdot,\cdot \gg$ as defined in \eqref{weaksym} is zero on the graph of $\sigma \cG$ if and only if
%Furthermore, the generating functional $\mathfrak{H}$ is \emph{nonnegative} if and only if $G\geq 0$.
%

\subsection{Reciprocity and passivity; relaxation systems}
\label{subsec:rec}
A key contribution of the seminal paper \cite{willems72} is the combination of reciprocity with \emph{passivity}. Recall that a system $\Sigma$ is passive if and only if there exists a matrix $Q=Q^\top \geq 0$ satisfying the (passivity) \emph{dissipation inequality} $\frac{d}{dt} \frac{1}{2} x^\top Q x \leq y^\top u$, or equivalently the Linear Matrix Inequality (LMI)
\bq
\label{paslinear}
\bma Q & 0 \\[2mm] 0 & I \ema \bma -A & -B \\[2mm] C & D \ema +
\bma -A^\top & C^\top \\[2mm] -B^\top & D^\top \ema  \bma Q & 0 \\[2mm] 0 & I \ema \geq 0.
\eq
The quadratic function $\frac{1}{2} x^\top Q x \geq 0$ is called a \emph{storage function} for the system $\Sigma$, and $Q$ is called a storage matrix. The system is called \emph{lossless} if $\frac{d}{dt} \frac{1}{2} x^\top Q x = y^\top u$, and thus \eqref{paslinear} is satisfied with equality. If $Q$ is indefinite, then the system is called \emph{cyclo-}passive, respectively \emph{cyclo-}lossless. The following result is well-known; see \cite{SIAMreciprocal} for a short proof.
\begin{proposition}
\label{prop:Q}
Suppose $Q=Q^\top$ is a solution to \eqref{paslinear}. Then $\ker Q$ is $A$-invariant and contained in $\ker C$. In particular, if the system is observable then necessarily $\ker Q=0$.
\end{proposition}
The storage matrix of a (cyclo-)lossless system is \emph{unique}, but in general a (cyclo-)passive system admits \emph{many} $Q$ satisfying the LMI \eqref{paslinear}. On the other hand, if $Q$ is an invertible solution to \eqref{paslinear}, and additionally the system is \emph{reciprocal} with respect to $G$ and $\sigma$, then, by combining \eqref{paslinear} and \eqref{G12}, one verifies that also $Q' := GQ^{-1}G$ is a solution of \eqref{paslinear}. By an application of Brouwer's fixed point theorem \cite{willems76}, or more direct methods \cite{willems72}, it follows that there exist $Q$ satisfying \eqref{paslinear} \emph{and}
\bq
\label{complinear}
Q = GQ^{-1}G.
\eq
A solution $Q$ of \eqref{paslinear} satisfying additionally \eqref{complinear} is said to be \emph{compatible} with $G$. 
Compatible $Q$ define storage functions with a clear physical relevance; cf. \cite{willems72}, \cite{willems76}, \cite{SIAMreciprocal}. 

In general, there may be many compatible $Q$. However, if $G>0$ then, as shown in \cite{ifac2011}, there is a \emph{unique} compatible storage matrix, namely $Q=G$. This case amounts to the definition of a \emph{relaxation system}\footnote{The original paper \cite{willems72} starts with an alternative definition of relaxation systems in terms of complete monotonicity of the impulse response matrix, which then is shown to be equivalent to the definition given here.}.
\begin{definition}
A \emph{relaxation} system is a passive system that is reciprocal with respect to the identity matrix $\sigma=I$ and with $G>0$.
\end{definition}
Thus relaxation systems $(A,B,C,D)$ are reciprocal systems with inner product $G>0$ satisfying the passivity dissipation inequality
\bq
\label{rela}
\frac{d}{dt} \frac{1}{2} x^\top G x \leq y^\top u.
\eq
\begin{remark}
{\rm In fact, a reciprocal system with $G\geq 0$, $A$ Hurwitz, and $D=D^\top \geq 0,$ is automatically passive. This comes from the fact that in this case the symmetric matrix $GA$ satisfies $GA \geq 0$; see \cite{chaffey23} for a proof. Using this property the inequality \eqref{rela} is immediately verified.
}
\end{remark}
Loosely speaking, relaxation systems are passive systems with only one type of energy (in the sense that no part of the energy is transformed into another), and therefore do not show "any hint of oscillatory behavior", cf. \cite{willems72}. For a broad range of physical examples of relaxation systems see \cite{willems72}, \cite{willems76}. 

\subsection{Port-Hamiltonian formulation of passive reciprocal systems}
The physical properties of passive reciprocal and relaxation systems become especially clear in their \emph{port-Hamiltonian} formulation \cite{jeltsema}; see also \cite{ifac2011}. 
%Start from the dissipation inequality \eqref{paslinear} for $Q\geq 0$. By Proposition \ref{prop:Q} $\ker Q \subset \ker Q$ and $\ker Q$ is invariant under $A$. Strengthen this last property to the additional assumption $\ker Q \subset \ker A$. Then there exists a matrix $F$ such that
%\bq
%\bma -A & -B \\[2mm] C & D \ema = F \bma Q & 0 \\[2mm] 0 & I \ema
%\eq
%($F$ is unique in case $Q$ is invertible.)
%Furthermore, in view of \eqref{paslinear}, one can choose $F$ such that $F + F^\top \geq 0$; cf. \cite{jeltsema}. Decompose $F$ into its skew-symmetric and symmetric positive semi-definite parts
%\bq
%\bma -J & -G \\[2mm] G^\top & M \ema, J=-J^\top, M=-M^\top, \bma R & P \\[2mm] P^\top & S \ema \geq 0
%\eq
%Then $\Sigma$ can be rewritten as the port-Hamiltonian system, with stored energy $\frac{1}{2}x^\top Qx$.
%\bq
%\begin{array}{rcl}
%\dot{x} & = & (J - R)Qx + (G-P)u \\[2mm]
%y &= & (G^\top + P^\top )Qx + (M+S)u
%\end{array}
%\eq
Consider a system $\Sigma=(A,B,C,D)$ which is reciprocal, i.e., satisfying \eqref{G12}. Write the system into its pseudo-gradient form \eqref{gradient}. Furthermore, let the system be passive, with invertible $Q$ satisfying \eqref{paslinear} and $Q=GQ^{-1}G$. For simplicity of exposition let $\sigma=I$; otherwise see \cite{SIAMreciprocal}. By Lemma 2.2 in \cite{ifac2011} there exist coordinates $x=(x_1,x_2)$ in which 
\begin{equation}\label{compatible1}
Q = \begin{bmatrix} Q_1 & 0 \\ 0 & Q_2 \end{bmatrix} \, , \quad G = \begin{bmatrix} Q_1 & 0 \\ 0 & - Q_2 \end{bmatrix}.
\end{equation}
In such coordinates the system takes the form, see \cite{ifac2011},
\begin{equation}\label{gradient2}
\begin{array}{rcl}
\begin{bmatrix} Q_1 & 0 \\ 0 & -Q_2 \end{bmatrix} \dot{x} & = & - Px +  \begin{bmatrix} C^\top_1 \\ 0 \end{bmatrix} u \\[4mm]
y & = & \begin{bmatrix} C_1 & 0 \end{bmatrix} x +Du, \quad D=D^\top \geq 0
\end{array} 
\end{equation}
Partition the symmetric matrix $P$ correspondingly as $P = \begin{bmatrix} P_1 & P_c \\ P_c^\top & P_2 \end{bmatrix}$. Then passivity implies $P_1=P_1^\top \geq 0,  P_2 = P_2^\top \leq 0$. Multiplying the second part (corresponding to $x_2$) of the differential equations in (\ref{gradient2}) on both sides with a minus sign, one obtains the equivalent system description
\begin{equation}\label{gradient3}
\begin{array}{rcl}
\begin{bmatrix} Q_1 & 0 \\ 0 & Q_2 \end{bmatrix} \dot{x} & = & - 
\begin{bmatrix} P_1 & P_c \\ -P_c^\top & -P_2 \end{bmatrix}x +  \begin{bmatrix} C^\top_1 \\ 0 \end{bmatrix} u \\[4mm]
y & = & \begin{bmatrix} C_1 & 0 \end{bmatrix} x +Du
\end{array} 
\end{equation}
with $P_1=P_1^\top \geq 0,  P_2 = P_2^\top \leq 0$.
Then in the new coordinates $z = (z_1,z_2)$, with $z_1=Q_1x_1$ and $z_2=Q_2x_2$, (\ref{gradient3}) takes the port-Hamiltonian form
\begin{equation}\label{portham1}
\begin{array}{rcl}
\dot{z} & = & \left( \begin{bmatrix} 0 & - P_c \\ P_c^\top & 0 \end{bmatrix} - \begin{bmatrix} P_1 & 0 \\ 0 & -P_2 \end{bmatrix} \right) \begin{bmatrix} Q_1^{-1} & 0 \\ 0 & Q_2^{-1}\end{bmatrix}z + 
\begin{bmatrix} C^\top_1 \\ 0 \end{bmatrix} u \\[6mm]
y & = & \begin{bmatrix} C_1 & 0 \end{bmatrix} \begin{bmatrix} Q_1^{-1} & 0 \\ 0 & Q_2^{-1}\end{bmatrix}z + Du,\quad D=D^\top \geq 0,
\end{array} 
\end{equation}
with Hamiltonian $ \frac{1}{2} z_1^\top Q_1^{-1}z_1 + \frac{1}{2}z_2^\top Q_2^{-1}z_2$. 

Thus the system is split into \emph{two energy domains}, corresponding to $z_1$ and $z_2$, which are only interconnected through the coupling matrix $P_c$. Physically, $z$ are the energy state variables, while $x=G^{-1}z$ are the co-energy state variables. For example, in an RLC electrical network with current sources, the components of $z_1$ are the \emph{charges} of the capacitors, and $z_2$ the \emph{flux linkages} of the inductors, while $x_1$ are the voltages across the capacitors and $x_2$ the currents through the inductors. Furthermore, $\frac{1}{2} x_1^\top Px_1$ is the content function of the combined current-controlled \emph{resistors}, and $\frac{1}{2} x_2^\top Px_2$ is the co-content of the voltage-controlled \emph{conductors}. Finally, $P_c$ is defined by the network topology (coupling capacitors to inductors, and thus electric energy to magnetic energy).

In the special case of a relaxation system $G=Q$, and the dynamics simplifies to
\bq
\label{recport}
\begin{array}{rcl}
\dot{z} &= &-PG^{-1}z + C^\top u \\[2mm]
y &=& CG^{-1}z + Du, \qquad D=D^\top \geq 0
\end{array}
\eq
with just one energy domain. 
For example, in an RL electrical network, the components of $z$ are flux linkages of the inductors and $x$ the currents through the inductors, while the potential function $\frac{1}{2} x^\top Px$ is the content of the current-controlled \emph{resistors}. 

%Alternatively, in an RC electrical network, the components of $z$ are \emph{charges} of the \emph{capacitors}, the components of $x$ are the voltages across the capacitors, and the potential function $\frac{1}{2} x^\top Px$ is the co-content of the combined voltage-controlled \emph{conductors}.

In this electrical network context the role of the signature matrix $\sigma$ is also clear. Indeed, the vectors $f$ and $e$ as occurring in the definition of a Lagrangian subspace (see Appendix 1) correspond to currents $I$, respectively voltages $V$, at the external ports of the network. It is well known that an impedance (from $I$ to $V$) or an admittance (from $V$ to $I$) representation is not always possible. However, a \emph{hybrid} input-output representation, from part of the currents together with a complementary part of the voltages, to the rest of the currents and voltages, \emph{is} possible. In the case of an RLCT electrical network this defines a reciprocal system with respect to the signature matrix defined by the splitting; see e.g. \cite{anderson}.

\section{Input-output Hamiltonian systems}
In this section we discuss \emph{another} class of systems with symmetry structure, namely linear \emph{input-output Hamiltonian systems} as originating from \cite{brockett,vds}. In this case the symmetry is reflected directly in the properties of the \emph{Volterra operator} of the system, instead of the Hankel operator as in the previous reciprocal case.

\subsection{Definitions, basic properties and geometric formulation}
\begin{definition}
A system $\Sigma$ is an \emph{input-output (IO) Hamiltonian system} with respect to a signature matrix $\sigma$ if its transfer matrix $K(s)$ satisfies $\sigma K(s)=K^\top (-s) \sigma$, or equivalently its impulse response matrix $W(t -\tau)$ satisfies $\sigma W(t-\tau)= - W^\top(\tau-t) \sigma$.
\end{definition}
As shown in \cite{vds}, a minimal state space system realization $\Sigma=(A,B,C,D)$ of $K(s)$ is an IO Hamiltonian system with respect to $\sigma$ if and only if there exists an invertible $\Omega=-\Omega^\top$ (which is unique) such that
\bq
\label{IOham}
A^\top \Omega + \Omega A =0, \; B^\top \Omega= \sigma C, \; \sigma D = D^\top \sigma.
\eq
The matrix $\Omega$ is also known as a \emph{symplectic form} on the state space $\X$, and there exist (so-called \emph{canonical}) coordinates in which
\bq
\Omega= \bma 0 & -I \\I & 0 \ema.
\eq
(In particular, the dimension of $\X$ is even.)

Input-output Hamiltonian systems are an extension to classical Hamiltonian systems; see e.g. \cite{brockett, vds}. They are related to the broad class of \emph{port-Hamiltonian systems} \cite{jeltsema, pass} as follows. Consider an IO Hamiltonian system $\Sigma$ with $D=0$, and for clarity of exposition let $\sigma=I$. Define the matrix $Q=\Omega A$, which by \eqref{IOham} and $\Omega=-\Omega^\top$ is symmetric. Then rewrite the equations of the IO Hamiltonian system as
\bq
\begin{array}{rcl}
\dot{x} & = & J Q x - J C^\top u \\[2mm]
y & = & Cx ,
\end{array}
\eq
with skew-symmetric matrix $J:=\Omega^{-1}$ (the Poisson structure matrix). Replacing the output $y$ by its \emph{time-derivative} $z=\dot{y}=CAx + CBu= C\Omega Qx - C\Omega C^\top u$ then leads to
\bq
\begin{array}{rcl}
\dot{x} & = & JQ x - JC^\top u \\[2mm]
z & = & CJ Qx - CJC^\top u
\end{array}
\eq
This is a \emph{port-Hamiltonian system} with Hamiltonian (stored energy) $\frac{1}{2}x^\top Q x$, satisfying the dissipation \emph{equality}
\bq
\frac{d}{dt} \frac{1}{2}x^\top Q x = x^\top Q \left(JQ x - JC^\top u \right) = z^\top u,
\eq
since $x^\top QJQx=0, u^\top CJC^\top u=0$.
\begin{example}
\label{ex:mass}
Consider a point mass with mass $m=1$, external force $u$, and output equal to the position of the mass. That is
\bq
A=\bma 0 & 1 \\ 0 & 0 \ema, \; B=\bma 0 \\ 1  \ema, \;  C=\bma 1 &  0 \ema, \; \Omega = \bma 0 & -1 \\ 1 & 0 \ema
\eq
This is an IO Hamiltonian system, with impulse response
\bq
W(t,\tau) =Ce^{At}\cdot e^{-A\tau}B = \bma -t & 1\ema \bma 0 & -1 \\ 1 & 0 \ema \bma -\tau \\[2mm] 1 \ema = t-\tau
\eq
%The constraints \eqref{compact} amount to
%\bq
%0=\int_\alpha^\beta u(\tau) d\tau, \quad 0 = \int_\alpha^\beta \tau u(\tau) d\tau.
%\eq
%Note that this system is \emph{not} cyclo-dissipative with respect to the supply rate $uy$. Indeed, its transfer function $K(s)= \frac{1}{s^2}$ does not admit a spectral factorization.
\end{example}
Thus input-output Hamiltonian systems satisfy the dissipation \emph{equality} with respect to the supply rate $u^\top z= u^\top \dot{y}$. Hence they are \emph{energy-conserving}. In \cite{nega} it has been shown how by extending input-output Hamiltonian systems to systems with energy dissipation they become equivalent to \emph{negative imaginary systems} \cite{petersen,lanzon} and \emph{counter-clockwise input-output systems} \cite{angeli}.

Being IO Hamiltonian can be also expressed by saying that the minimal state space system $\Sigma =(A,B,C,D)$, with inputs $u$ and outputs $\sigma y$, has the same input-output behavior as its \emph{adjoint system} 
\bq
\label{adjoint}
\Sigma^a: \quad \begin{array}{rcl}
\dot{p} & = & - A^\top p - C^\top  u^a
\\[2mm]
y^a & = & B^\top p +  D^\top  u^a
\end{array}
\eq
with inputs $u^a=\sigma u$ and outputs $y^a$. In fact, the unique state space isomorphism between $\Sigma$ and $\Sigma^a$ is given as $p=\Omega x$. In this sense IO Hamiltonian systems are \emph{self-adjoint}; see \cite{crouchvds}.

Recall that the adjoint system $\Sigma^a$ is characterized by the property
\bq
\frac{d}{dt} p^\top (t) x(t) = \left(y^a(t)\right)^\top u(t) - \left(u^a(t) \right)^\top y(t).
\eq
Using $p=\Omega x$ and $\Omega^\top=-\Omega$, this leads to the following defining identity of IO Hamiltonian systems
\bq
\frac{d}{dt} x_2^\top (t) \Omega x_1(t) = u_2^\top (t) y_1(t) - y_2^\top (t) u_1(t),
\eq
for all solution triples $(x_i,u_i,y_i), i=1,2$.

\subsection{Volterra operators of input-output Hamiltonian systems}
Equations \eqref{IOham} mean that the impulse response matrix of an IO Hamiltonian system can be rewritten as
\bq
\begin{array}{rcl}
\sigma W(t-\tau) &= & B^\top J e^{At} e^{-A\tau}B + \sigma D \delta (t-\tau) \\[2mm]
&=& B^\top e^{-A^\top t} J e^{-A\tau}B + D^\top \sigma \delta (t-\tau) .
\end{array}
\eq
For simplicity of exposition let us take throughout the rest of this section $D=0$ and $\sigma=I$. Then the Volterra integral operator is the map
\bq
\V: L_2(-\infty,\infty) \to L_2(-\infty,\infty), \quad u(\cdot) \mapsto y(\cdot) ,
\eq
given by
\bq
y(t)= B^\top e^{-A^\top t} \Omega \int_{-\infty}^t e^{-A\tau}Bu(\tau) d\tau.
\eq
In order that $\V$ is well-defined we will \emph{restrict} the domain of $\V$ to the set of input functions $u:(-\infty,\infty) \to \mR^m$ such that (1) $u$ has compact support, (2) the corresponding output $y$ has compact support. This second condition is equivalent to existence of $\alpha, \beta$ such that
\bq
\label{compact}
\int_\alpha^\beta e^{-A\tau}Bu(\tau) d\tau =0
\eq
(which ensures that the support of $y$ is within $(\alpha,\beta)$).

We see three major differences between the IO Hamiltonian case and the previous, reciprocal, case: (1) Volterra instead of Hankel, (2) $A$ does not need to be Hurwitz, (3) because of skew-symmetry of $\Omega$ the kernel
\bq
\label{kernel}
B^\top e^{-A^\top t} \Omega  e^{-A\tau}B, \; \tau < t, \qquad 0  \mbox{ whenever } \tau > t,
\eq
\emph{appears} to be skew-symmetric for $t > \tau$, instead of the symmetry that we encountered in the reciprocal case. However, as will become clear, the Volterra operator is \emph{not} skew-symmetric, but in fact \emph{symmetric}; precisely because of the compact support assumptions on its domain.

We start from the following proposition proven in \cite{vdslyon}, specializing a main result of \cite{crouchvds} to the linear case. 
%For simplicity, in the rest of this section take $\sigma=I$ (\textcolor{red}{general case to be stated later on}).
\begin{proposition}
\label{propIOham}
Consider a system $\Sigma=(A,B,C,D)$ with equally dimensioned inputs and outputs $u,y$ in $\mR^m$. Consider a signal pair $u^a,y^a$ on $(-\infty,\infty)$ with compact support. Then
\bq
\int_\infty^\infty \left(u^a(t)\right)^\top y(t) dt = \int_\infty^\infty \left(y^a(t)\right)^\top u(t) dt
\eq
for all input-output trajectories $(u,y)$ of the given system $(A,B,C,D)$ with compact support, if and only if $(u^a,y^a)$ is an input-output trajectory of the \emph{adjoint} system $(-A^\top,-C^\top,B^\top,D^\top)$.
\end{proposition}
%It follows from the proof of this proposition that the subspace of all such inputs are the compact support inputs $u$ satisfying
%\bq
%\label{compact}
%\int_\infty^\infty e^{-A\tau}Bu(\tau) d\tau =0
%\eq
%(This defines a dense subspace of $L_2(-\infty,\infty)$.) 
%The subspace of input-output trajectories $(u,y)$ of \emph{compact support}, i.e., $u$ of compact support and satisfying \eqref{compact}, will be denoted by $\cS$.

This proposition immediately leads to the following characterization of linear IO Hamiltonian systems; see \cite{crouchvds} for further background and extensions to the nonlinear case.
\begin{proposition}
A minimal state space system $\Sigma=(A,B,C,D)$ is IO Hamiltonian if and only if the input-output behavior of compact support of $(A,B,C,D)$ is the same as that of its adjoint system $(-A^\top,-C^\top,B^\top,D^\top)$.
\end{proposition}
From a geometric point of view this means that the set of input-output trajectories $(u,y)$ of compact support of an input-output Hamiltonian system defines a Lagrangian subspace $\cL$ of the set of all signals (not necessarily input-output trajectories) $(u,y)$ of compact support. Equivalently, the Volterra $\V$ operator \emph{restricted} to the space of input functions of compact support and satisfying \eqref{compact} is \emph{self-adjoint}. 

Proposition \ref{propIOham} tells us that for an IO Hamiltonian system the bilinear form
\bq
\label{symmetric}
\left( (u_1,y_1), (u_2,y_2) \right):= \int_\infty^\infty u_1^\top(t)y_2(t) dt 
\eq
is \emph{symmetric} when restricted to the set of input-output trajectories of compact support. Thus although the kernel $B^\top e^{-A^\top t} J\Omega e^{-A\tau}B$ \emph{looks} skew-symmetric, the bilinear form \eqref{symmetric} is actually \emph{symmetric} in view of the constraints \eqref{compact}.
The generating functional of this Lagrangian subspace $\cL$ is the functional $\mathfrak{V}: L_2(-\infty,\infty) \to \mR$ defined as
\bq
\begin{array}{rcl}
\mathfrak{V}(u) &:= &\frac{1}{2} \int_{-\infty}^\infty u^\top (t) y(t) dt \\[2mm]
&= & \frac{1}{2} \int_{-\infty}^\infty u^\top (t) B^\top e^{-A^\top t} \Omega \left( \int_{-\infty}^t  e^{-A\tau}Bu(\tau)d\tau \right) dt.
\end{array}
\eq
Indeed, using symmetry of the bilinear form \eqref{symmetric}, it follows that the variational derivative of $V(u)$ is the output function $y$.

\subsection{Nonnegative input-output Hamiltonian systems}
A special case occurs if the generating functional $\mathfrak{V}$ is \emph{nonnegative}, i.e., 
\bq
\label{cycloIO}
\int_{-\infty}^{\infty} u^\top (t)y(t) dt \geq 0
\eq
for all input-output trajectories $(u(\cdot),y(\cdot))$ of compact support. 
%(This can be considered as the counterpart of the property $G\geq 0$ in the case of reciprocal systems.)
Property \eqref{cycloIO} is the same as \emph{cyclo-passivity} of the IO Hamiltonian system. It means \cite{SIAMspectral} that there exists a (possibly indefinite) storage function $\frac{1}{2}x^\top W x$, with $W=W^\top$, satisfying
\bq
\label{cyclodiss}
\frac{d}{dt} \frac{1}{2}x^\top W x \leq u^\top y.
\eq
As shown in \cite{SIAMspectral}, if an invertible $W$ is a solution to the dissipation inequality \eqref{cyclodiss} then $\Omega W^{-1} \Omega$ is a solution as well. Since the set of $W$ satisfying \eqref{cyclodiss} is convex and compact, and the map $W \mapsto \Omega W^{-1} \Omega$ is continuous, it follows from Brouwer's fixed point theorem that there exist $W$ satisfying \eqref{cyclodiss} as well as
%\bq
%\label{brouwerS}
$W= \Omega W^{-1}\Omega$.
(The same argument was used for proving that in the reciprocal case there exists $Q$ satisfying \eqref{G12} and $Q= GQ^{-1}G$ \cite{willems76}.)
As discussed in \cite{SIAMspectral} this, in turn, implies that
\bq
(\Omega^{-1}W)^\top \Omega (\Omega^{-1}W) = -\Omega, \quad (\Omega^{-1}W)^2=I
\eq
Said otherwise, $\Omega^{-1}W$ is an anti-symplectomorphism and an involution. As follows from the results in \cite{meyer}, cf. \cite{SIAMspectral}, this means that there exist linear coordinates $x=(q,p)$ in which
\bq
\Omega= \bma 0 & -I \\ I & 0 \ema, \quad W = \bma 0 & I \\ I & 0 \ema,
\eq
i.e., the storage function is given by $p^\top q$. Note that this corresponds to the following appealing form of the dissipation inequality
\bq
\frac{d}{dt} \bma q^\top & p^\top \ema \bma  0 & I \\ I & 0 \ema \bma q \\[2mm] p \ema \leq \bma u^\top & y^\top \ema \bma  0 & I \\ I & 0 \ema \bma u \\ y \ema.
\eq
%Writing out the dissipation inequality 
%$\frac{d}{dt} \frac{1}{2}x^\top S x \leq u^\top y$ 
In such coordinates $(q,p)$ the system $\Sigma = (A,B,C)$, for simplicity with $D=0$, takes the form
\bq
\begin{array}{l}
A=\bma F & -P \\[2mm] -S & -F^\top \ema, \quad P=P^\top \geq 0, \quad S=R^\top \geq 0 \\[5mm]
B= \bma 0 \\[2mm] H^\top \ema, \quad C= \bma H & 0 \ema.
\end{array}
\eq
Next step is to consider the solution $X=X^\top$ of the Riccati equation
\bq
F^\top X + XF -XPX +S=0
\eq
(Controllability of $(A,B)$ implies controllability of $(F,P)$; cf. \cite{SIAMspectral}.) Application of the canonical transformation
\bq
\bma I  & 0 \\ - X & I \ema
\eq
to $(A,B,C)$ yields the transformed system $(\widetilde{A},\widetilde{B},\widetilde{C})$ given by
\bq
\widetilde{A}=\bma F-PX & -P \\ 0 & -F^\top \ema, \quad
\widetilde{B}= \bma 0 \\ H^\top \ema, \quad \widetilde{C}= \bma H & 0 \ema
\eq
It follows that the transfer matrix $K(s)=C(Is - A)^{-1}B$ factorizes as $K(s) = M(s)M^\top(-s)$, where
\bq
M(s) := H\left(sI - (F-PX)\right)^{-1} G \quad P=GG^\top
\eq
Thus nonnegativity corresponds to \emph{factorizability}. 
%From a state space point of view factorizability means that the IO Hamiltonian system can be factorized as the series interconnection of the system with transfer matrix $M(s)$ and its adjoint system with transfer matrix $M^\top (-s)$.
\begin{example}
{\rm
Consider the point-mass from Example \ref{ex:mass}, with impulse response $W(t,\tau) = t-\tau$. The constraints \eqref{compact} amount to
\bq
0=\int_\alpha^\beta u(\tau) d\tau, \quad 0 = \int_\alpha^\beta \tau u(\tau) d\tau.
\eq
This system is \emph{not} cyclo-dissipative with respect to the supply rate $uy$, since its transfer function $K(s)= \frac{1}{s^2}$ does not admit a factorization. In fact its generating functional $\mathfrak{V}$ is \emph{indefinite}.
}
\end{example}

\section{Time-reversibility, and its connections to reciprocal and IO Hamiltonian systems}

In this section we explore another type of symmetry, namely time-reversibility, in relation with the previous symmetry structures of reciprocity and being IO Hamiltonian. Passivity will again play a major role as well.

\subsection{Reciprocal systems and signed time-reversibility}
First, let us investigate the combination of reciprocity and passivity with the notion of \emph{signed time-reversibility}.
\begin{definition}
A system \eqref{system} is signed time-reversible if its transfer matrix $K(s)$ satisfies $K(s)=-K(-s)$.
\end{definition}
Equivalently, a system is signed time-reversible if whenever the pair $(u(t),y(t)), t \in \mR,$ is in the input-output behavior, then so is $(-u(-t),y(t)),$\\$ t \in \mR$. (Note the minus sign in front of the input.) It is shown in \cite{willemsrev} that a minimal system $\Sigma=(A,B,C,D)$ is signed time-reversible if and only if there exists an invertible $R: \X \to \X$ (the time-reversibility map), which is \emph{unique}, such that 
\bq
\label{stimereversible}
RA= -AR, \quad RB= B, \quad C=CR.
\eq
Furthermore, if $R$ satisfies these equations, then so does $R^{-1}$, and thus by uniqueness $R^2=I$; i.e., the signed time-reversibility map $R$ is an \emph{involution}. 

Now consider a minimal \emph{reciprocal} system (for simplicity with respect to $\sigma=I$), which is also signed time-reversible. Then its transfer matrix $K(s)$ satisfies $K(s)= -K(-s)= - K^\top (-s)$. This means that the system is also \emph{cyclo-lossless} with a storage function $\frac{1}{2}x^\top Q x$, where $Q=Q^\top$ is invertible, but not necessarily positive definite. 
Conversely, if a system is cyclo-lossless with storage matrix $Q$, then \eqref{paslinear} with equality yields
\bq
\label{conservative}
A^\top Q + QA  =0, \quad B^\top Q = C, \quad D + D^\top =0
\eq
This leads to the following proposition.
\begin{proposition}
\label{prop:4.2}
A minimal system that is reciprocal and signed time-reversible is also cyclo-lossless. Conversely, if the system $\Sigma=(A,B,C,D)$ is cyclo-lossless with invertible storage matrix $Q$, and reciprocal with pseudo-inner product $G$, then $D=0$ and $Q=GQ^{-1}G$, while the system is signed time-reversible with signed time-reversibility map $R:= Q^{-1}G$. 
\end{proposition}
\begin{proof}
$D + D^\top =0$ and $D=D^\top$ obviously implies $D=0$. It is immediately checked that $GQ^{-1}G$ satisfies \eqref{conservative}. Hence by uniqueness of $Q$ it follows that $Q=GQ^{-1}G$. Thus $R=Q^{-1}G$ is an involution, and one verifies that $R$ satisfies \eqref{stimereversible}.
\end{proof}
\begin{remark}
{\rm
From a coordinate-free perspective, $G$ and $Q$ define linear maps $G: \X \to \X^*$ and $Q: \X \to \X^*$ (with $\X^*$ the dual space), and $R$ defines an involution map $R: \X \to \X$. These maps $G,Q$ and $R$ satisfy the \emph{commutativity} relations $R=Q^{-1}G=G^{-1}Q$.
}
\end{remark}
\begin{remark}
{\rm
Conversely it can be shown that cyclo-passivity together with signed time-reversibility implies reciprocity and cyclo-losslessness; see the corresponding statement for the non-cyclo case in \cite{willems72}.
}
\end{remark}
A strengthened version of Proposition \ref{prop:4.2} is obtained whenever $Q>0$, i.e., the system is lossless. Indeed, if $Q>0$ and the system is reciprocal with pseudo-inner product $G$, then there are coordinates $x=(x_1,x_2)$ such that $Q$ and $G$ take the form \eqref{compatible1}, and thus $R= \bma I & 0 \\0 & -I \ema$.
Rewriting the system in pseudo-gradient system form \eqref{gradient} it follows  that the system takes the form (compare with \eqref{gradient3})
\bq
\begin{array}{rcl}
\bma Q_1 & 0 \\[2mm] 0 & -Q_2 \ema \dot{x} & = & - \bma 0 & P_c \\[2mm] P_c^\top & 0 \ema x + \bma C_1^\top \\[2mm] 0 \ema u\\
y & = & \bma C_1 & 0 \ema x +Du
\end{array}
\eq
This is a lossless system with two types of energy $\frac{1}{2}x_1^\top Q_1 x_1$ and $\frac{1}{2}x_2^\top Q_2 x_2$, and power-conserving interconnection structure defined by $P_c$. An example is provided by LC electrical networks, where $x_1$ refer to capacitor voltages and $x_2$ to inductor currents.
\begin{remark}
{\rm
It follows that \emph{relaxation systems} are typically \emph{not} signed time-reversible. In fact, since $G>0$, a relaxation system is signed time-reversible if and only if it is of the \emph{integrator} form (and thus energy-conserving)
\bq
G \dot{x} = C^\top u, \; G=G^\top >0, \qquad y=Cx.
\eq
%Being a system without energy dissipation this is \emph{not} corresponding to the standard interpretation of a relaxation system, i.e., a system with one type of energy with pervasive energy dissipation.
}
\end{remark}

\subsection{IO Hamiltonian systems and time-reversibility}
Throughout this subsection we take $D=0$ and $\sigma=I$ for ease of exposition. 
As noted before, the transfer matrix $K(s)$ of an IO Hamiltonian system is characterized by the property $K(s)=K^\top (-s)$. On the other hand, reciprocity amounts to $K(s)=K^\top (s)$. Taken together this implies $K(s)=K^\top(-s)= K(-s)$. This last property corresponds to \emph{time-reversibility}: whenever $(u(t),y(t)), t \in \mR,$ is an input-output trajectory of the system, then so is the time-reversed version $(u(-t),y(-t)), t \in \mR$. (Note that this is \emph{different} from \emph{signed} time-reversibility, where we needed a minus sign in front of the time-reversed input function.) Similar reasoning as for signed time-reversibility of reciprocal systems implies that among the three notions of $1)$ IO Hamiltonian, $2)$ reciprocity, $3)$ time-reversibility, \emph{two} of these three notions actually \emph{imply} the third one.

From a \emph{state space} point of view this is seen as follows. For simplicity of exposition let $\Sigma$ be given by a minimal triple $(A,B,C)$ with $D=0$. Being \emph{IO Hamiltonian} corresponds by \eqref{IOham} to the existence of a unique invertible $\Omega$ such that \bq
\label{1}
A^\top \Omega + \Omega A=0,  \; B^\top \Omega=C,  \; \Omega=-\Omega^\top.
\eq
On the other hand, \emph{reciprocity} amounts to the existence of a unique invertible $G=G$ such that
\bq
\label{2}
A^\top G=GA, \; B^\top G=C, \; G=G^\top.
\eq
Finally, \emph{time-reversibility} amounts to the existence of a unique invertible $R$ such that
\bq
\label{3}
RA=-AR, \; RB=-B, \; C=CR, \; R=R^{-1}.
\eq
(Note the extra minus sign in front of $B$ as compared to the equations \eqref{stimereversible} defining \emph{signed} time-reversibility.) 
\begin{proposition}
Consider a minimal system $(A,B,C)$. Then
\begin{enumerate}
\item
Suppose the system satisfies \eqref{1} and \eqref{2}. Then $G \Omega^{-1}G=\Omega$, and $R:=\Omega^{-1}G$ satisfies \eqref{3}.
\item
Suppose the system satisfies \eqref{2} and \eqref{3}. Then $R^\top G R =-G$ and $\Omega:=GR$ satisfies \eqref{1}.
\item
Suppose the system satisfies \eqref{3} and \eqref{1}. Then $R^\top \Omega R=-\Omega$ and $G:=\Omega R$ satisfies \eqref{2}.
\end{enumerate}
Hence if the system satisfies two out of the three equations \eqref{1}, \eqref{2}, \eqref{3}, then it also satisfies the third. 
\end{proposition}
\begin{proof}
If $\Omega$ and $G$ satisfy \eqref{1}, respectively \eqref{2}, then also $G \Omega^{-1}G$ satisfies \eqref{1}. Hence by uniqueness $G \Omega^{-1}G=\Omega$. It is immediately checked that $R:=\Omega^{-1}G$ satisfies \eqref{3}. This proves the first statement. Proof of the other two statements is analogous.
\end{proof}
Now suppose the IO Hamiltonian system $\Sigma=(A,B,C)$ is time-reversible (and therefore also reciprocal). In view of $R^\top \Omega R=-\Omega$ this means, see \cite{meyer,timereversible}, that there exists a basis for $\X$ in which
\bq
\label{basisJR}
\Omega= \bma 0 & -I \\ I & 0 \ema, \; R= \bma I & 0 \\ 0 & -I \ema, \; \mbox{and therefore } G= \bma 0 & I \\ I & 0 \ema.
\eq
In such a basis with coordinates $x=(q ,p)$ the system takes the form \cite{timereversible}
\bq
\begin{array}{rcl}
\bma \dot{q} \\ \dot{p} \ema & = & \bma 0 & P \\ -Q & 0 \ema + \bma q \\[2mm] p \ema + \bma 0 \\ \tilde{B} \ema u \\[3mm]
y & = & \bma \tilde{B}^\top & 0 \ema \bma q \\ p \ema
\end{array}
\eq
for certain matrices $P=P^\top$, $Q=Q^\top$, and $\tilde{B} $.

%Finally, let us consider the relation between \emph{nonnegativity} of the transfer matrix $K(s)$ of a minimal IO Hamiltonian system $(A,B,C)$ (as discussed in the previous subsection) and \emph{time-reversibility}. Recall that nonnegativity of $K(s)$ implies the existence of an invertible $W=W^\top$ satisfying
%\bq
%\label{W}
%A^\top W + W A \leq 0, \quad B^\top W = C
%\eq
%Using \eqref{3} it is verified that then also $-R^\top WR$ satisfies the same (in)equalities. Therefore, by similar arguments as before, there exist $W=W^\top$ satisfying \eqref{W} \emph{and} $W=-R^\top WR$. In a basis in which \eqref{basisJR} holds this means
%\bq
%W = \bma 0 & \tilde{W} \\[2mm]  \tilde{W}^\top & 0 \ema
%\eq
%for some $\tilde{W}$.

\section{Conclusions and outlook}
An overview has been given of linear systems endowed with symmetry structures, with emphasis on reciprocal and input-output Hamiltonian systems. Striking feature is the interplay between state space, input-output, and geometric formulations. Partial extensions to the nonlinear case already exist, see e.g. \cite{SIAMreciprocal,sepulchre24} for the reciprocal case and \cite{crouchvds} for the input-output Hamiltonian case, but many open questions remain.

Apart from the nonlinear generalization an important challenge is how to use the developed theory for design and control. In particular, how to use (linear or nonlinear) systems with symmetry as building blocks in applications such as neuromorphic computation.

\begin{appendices}

\section{Lagrangian subspaces and Dirac structures}\label{secA1}
Let $\F$ be a linear space. Denote by $\E:=\F^*$ its dual space, with $<e,f>$ the (algebraic) duality product between $f \in \F$ and $e \in \E$. Then $\F \times \E$ is endowed with the canonical \emph{symplectic form}
\bq
\label{asym}
\langle (f_a,e_a),(f_b,e_b) \rangle =  <e_a, f_b> - <e_b,f_a>.
\eq
In case $\F=\mR^n$ the symplectic form amounts to the matrix
\bq
\label{J}
J=\bma 0 & -I_n \\ I_n & 0 \ema
\eq
A subspace $\cL \subset \F \times \E$ is \emph{Lagrangian} if the symplectic form $\langle \cdot, \cdot \rangle$ is \emph{zero} restricted to $\cL$, and furthermore $\cL$ is \emph{maximal} with respect to this property. In case $\F$ is an $n$-dimensional linear space this means $\dim \cL=n$. Conversely, any $n$-dimensional subspace $\cL$ on which $\langle \cdot, \cdot \rangle$ is zero is Lagrangian. Furthermore
\begin{proposition}
A subspace $\cL \subset \F \times \E$ is Lagrangian if and only if $\cL = \cL^{\pperp}$, where ${}^{\pperp}$ denotes orthogonal companion with respect to the symplectic form \eqref{asym}.
\end{proposition}
Let $\cL \subset \F \times \E$ be Lagrangian. If there exists a mapping $S: \F \to \E$ such that $\cL$ is the graph of $S$, then $S$ is \emph{symmetric}. Conversely, for any symmetric $S: \F \to \E$ the subspace $\mbox{graph } S$ is Lagrangian. Similarly, if there is a mapping $S: \E \to \F$ such that $\cL$ is the graph of $S$.

If $\cL$ can\emph{not} be parametrized, either by $\F$ or by $\E$, we obtain the following general result in case $\F$ is finite-dimensional. Let $\F=\mR^n$. Denote the linear coordinates for $\F$ and $\E$ by $f_1, \cdots,f_n$ and $e_1, \cdots, e_n$. Then for any Lagrangian subspace $\cL$ there exists a partitioning $\{1, \cdots,n\} = I_1 \cup I_2$ such that $\cL$ is parametrized by the coordinates $f_i, i \in I_1,$ and $e_i, i \in I_2$. Write (possibly after joint permutations in $f$ and $e$)
\bq
f= \bma f^1 \\f^2 \ema, \quad e= \bma e^1 \\e^2 \ema,
\eq
where the components of $f^1$ are $f_i, i \in I_1,$ and of $f^2$ are $f_i, i \in I_2$. Same for $e^1$ and $e^2$. Thus there exists an $n \times n$ matrix $S$ such that
\bq
\label{LagS}
\cL = \{ \bma f^1 \\f^2 \\ e^1 \\e^2 \ema \mid \bma e^1 \\f^2 \ema = S \bma f^1 \\e^2 \ema \}
\eq
Then $S$ satisfies
\bq
\label{S}
\bma I_{n_1} & 0 \\ 0 & -I_{n_2} \ema S = S^\top \bma I_{n_1} & 0 \\ 0 & -I_{n_2} \ema,
\eq
where $n_1$, respectively $n_2$, is the cardinality of $I_1$, respectively $I_2$. Conversely, for any $S$ satisfying \eqref{S} the subspace \eqref{LagS} is a Lagrangian subspace. Note that $\bma I_{n_1} & 0 \\ 0 & -I_{n_2} \ema$ is a \emph{signature matrix}. The quadratic function
\bq
\bma (f^1)^\top & (e^2)^\top \ema S \bma f^1 \\e^2 \ema
\eq
is called a \emph{generating function} of the Lagrangian subspace \eqref{LagS}. 

On the other hand, \emph{Dirac structures} can be regarded as 'skew-symmetric' counterparts of Lagrangian subspaces. Instead of \eqref{sym} consider the following bilinear form on $\F \times \E$ (replace $-$ by $+$)
\bq
\label{sym}
[[ (f_a,e_a),((f_b,e_b) ]] :=  <e_a, f_b> + <e_b,f_a>
\eq
%In case $\F=\mR^n$ this symmetric bilinear form amounts to the matrix
%\bq
%\label{+}
%\bma 0 & I_n \\ I_n & 0 \ema
%\eq
A subspace $\D \subset \F \times \E$ is a (constant) \emph{Dirac structure} if the bilinear form $[[ \cdot, \cdot ]]$ is \emph{zero} restricted to $\D$, and furthermore $\D$ is \emph{maximal} with respect to this property. In case $\F$ is an $n$-dimensional linear space this means that $\dim \D=n$. Conversely, any $n$-dimensional subspace $\D$ on which $[[ \cdot, \cdot ]]$ is zero is a Dirac structure. Furthermore
\begin{proposition}
A subspace $\D \subset \F \times \E$ is a Dirac structure if and only if $\D = \D^{\pperp}$, where ${}^{\pperp}$ denotes orthogonal companion with respect to the symmetric bilinear form \eqref{sym}.
\end{proposition}
Let $\D \subset \F \times \E$ be a Dirac structure. If there exists a mapping $Z: \F \to \E$ such that $\D$ is the graph of $Z$, then $Z$ is \emph{skew-symmetric}. Conversely, for any skew-symmetric $Z: \F \to \E$ the subspace $\mbox{graph } Z$ is a Dirac structure. Similarly, if there is a mapping $Z: \E \to \F$ such that $\D$ is the graph of $Z$.

\subsection{Static reciprocity and losslessness}
\label{recloss}
Roughly speaking, static \emph{lossless} structures correspond to \emph{skew-symmetric} matrices, while \emph{reciprocal} structures correspond to \emph{symmetric} matrices. In this sense losslessness and reciprocity could be expected to \emph{exclude each other}. The purpose of this subsection is to show that this is \emph{not} necessarily the case if we consider reciprocity \emph{with respect to a signature matrix}. 

First recall from \cite{SIAMgraphs} that a Dirac structure $\D \subset \F \times \F^*$, with $\F$ a linear space, is called \emph{separable} if
$\D = \K \times \K^\perp$
for some subspace $\K \subset \F$. Furthermore, as shown in \cite{SIAMgraphs} that a Dirac structure $\D$ is separable if and only if
\bq
\label{tellegen}
<e_b, f_a> = 0
\eq
for all pairs $(f_a,e_a), (f_b,e_b) \in \D \subset \F \times \F^*$. 
A typical example of a separable Dirac structure is the space of currents and voltages constrained by Kirchhoff's current and voltage laws. It is separable because Kirchhoff's \emph{current} laws are decoupled from Kirchhoff's \emph{voltage} laws. The expression \eqref{tellegen} in this context is Tellegen's theorem $<V_1, I_2>=0$ for all $V_1$ satisfying Kirchhoff's voltage laws, and all $I_2$ satisfying Kirchhoff's current laws (where $V_1$ and $I_2$ need not be the \emph{actual} voltages and currents). 
Since transformers do not mix voltages and currents either, also Kirchhoff's current and voltage laws \emph{together with} transformers define separable Dirac structures. 

We have the following alternative characterization of separable Dirac structures.
\begin{proposition}
A Dirac structure $\D \subset \F \times \F^*$ is separable if and only if it is Lagrangian.
\end{proposition}
\begin{proof}
Let $\D$ be a separable Dirac structure. Take any two $(f_a,e_a),$$ (f_b,e_b) \in \D$. Then by \eqref{tellegen} $<e_a, f_b> = 0, <e_b, f_a> = 0$. Hence $<e_a, f_b>  - <e_b, f_a>=0$, and thus $\D$ is Lagrangian. Conversely, let $\D$ be a Dirac structure that is also Lagrangian. Since $\D$ is a Dirac structure $<e_a, f_b> + <e_b, f_a>=0$ for any two $(f_a,e_a), (f_b,e_b) \in \D$. Since $\D$ is Lagrangian also $<e_a, f_b>  - <e_b, f_a>=0$, implying that $<e_a, f_b> =0, <e_b, f_a>=0$. Hence $\D$ is separable.
%\eqref{J}, \eqref{+}
\end{proof}
From an explicit equational point of view the equivalence between separability of $\D$ and $\D$ being Lagrangian is seen as follows. For any Dirac structure $\D \subset \F \times \E$ there exist square matrices $F,E$, satisfying $FE^\top + EF^\top =0$, such that $\D = \ker \bma F & E \ema$; cf. \cite{jeltsema}. Furthermore, cf. \cite{bloch}, modulo column permutations, one can always split $F= \bma F_1 & F_2 \ema$ and correspondingly $E= \bma E_1 & E_2 \ema$ such that $\bma F_1 & E_2 \ema$ is invertible. Denote by $f= \bma f^1 \\ f^2 \ema$ and $e= \bma e^1 \\ e^2 \ema$ the corresponding splitting of the vectors $f,e$. 

Now let $\D$ be separable, i.e. $\D = \K \times \K^\perp$ for some subspace $\K \subset \F$. Then it follows that modulo row permutations $F$ and $E$ take the form
\bq
F = \bma  \bar{F}_1 & \bar{F}_2 \\ 0 & 0  \ema, \quad E = \bma  0 & 0 \\\bar{E}_1 & \bar{E}_2 \ema,
\eq
where $\K= \ker \bma  \bar{F}_1 & \bar{F}_2 \ema$. It follows that
\bq
\bma f^1 \\[2mm] e^2 \ema = \bma \bar{F}_1 & 0 \\[2mm] 0 & \bar{E}_2 \ema^{-1} \bma 0 & \bar{F}_2  \\[2mm] \bar{E}_1 & 0 \ema \bma e^1 \\[2mm] f^2 \ema = \bma 0 & G \\[2mm] -G^\top & 0\ema \bma e^1 \\[2mm] f^2 \ema,
\eq
for some matrix $G$.

Clearly the skew-symmetric matrix $J:=\bma 0 & G \\-G^\top & 0\ema$ satisfies
\bq
\bma I & 0 \\ 0 & -I \ema J = J^\top \bma I & 0 \\ 0 & -I \ema,
\eq
and thus is \emph{reciprocal} with respect to the signature matrix $\bma I & 0 \\ 0 & -I \ema$. Equivalently, $\D$ is Lagrangian with generating function $\frac{1}{2}(e^1)^\top G f^2$.

\section{Hankel kernel analysis of reciprocal systems}\label{secA2}
The symmetric kernel \eqref{symkernel} associated with the Hankel operator $\cH$ of any reciprocal system $\Sigma=(A,B,C,D)$ can be analyzed as follows.
Since the rank of $\cH$ is less than or equal to $n$, the dimension of the state space, the spectrum of $\sigma \cH$ is of the form
\bq
\{\lambda_1, \cdots, \lambda_n \} \cup \{0 \}, \quad \lambda_i \in \mC.
\eq
In fact, as shown in \cite{sorensen} (see already \cite{fernando1} and \cite{fernando2} for the SISO case), 
the eigenvalues $\lambda_1, \cdots, \lambda_n$ are equal to the eigenvalues of the \emph{cross-Gramian} $Z$, defined as the unique solution of the Sylvester equation $AZ + ZA= - B \sigma C$.
Furthermore, since the signature modified Hankel operator $\sigma \cH$ of a reciprocal system is \emph{self-adjoint} its eigenvalues $\lambda_1, \cdots, \lambda_n$ are \emph{real} (and equal to the singular values of $\sigma \cH$). 
%and eigenfunctions $\phi_1, \cdots, \phi_n$. 
%Furthermore, as shown in \cite{sorensen}, see already \cite{fernando1} for the SISO case as well as \cite{fernando2}, the set of eigenvalues
This leads to the following explicit expression of the eigenvalues $\lambda_1, \cdots, \lambda_n$ of $\sigma \cH$ together with their eigenfunctions $\phi_1, \cdots, \phi_n$. Consider for the eigenfunctions the \emph{ansatz} $\phi (t):= B^\top e^{A^\top t} \C^{-1} x$,
with $\C:= \int_0^{\infty} e^{A\tau}BB^\top e^{A^\top \tau} d \tau$ the \emph{controllability Gramian} of the pair $(A,B)$. One obtains
\bq
\sigma \left( \cH \phi \right)(t) = B^\top e^{A^\top t} G \left(\int_0^{\infty} e^{A\tau}BB^\top e^{A^\top \tau} d \tau \right) \C^{-1}x = 
B^\top e^{A^\top t} Gx.
\eq
It follows that $\phi_i(t):= B^\top e^{A^\top t} \C^{-1} x_i$ is an eigenfunction corresponding to $\lambda_i$ if and only if $Gx_i=\lambda_i \C^{-1}x_i$, or equivalently $\C G x_i= \lambda_i x_i$.
Now, as shown in \cite{ionescu,scherpen}, $\C G$ equals the cross-Gramian $Z$.
Thus the eigenfunctions $\phi_i$ are generated by eigenvectors $x_1, \cdots,x_n$ of $\C G$ corresponding to the real eigenvalues $\lambda_1, \cdots,\lambda_n$ of $\C G$. 
\begin{remark}
{\rm Furthermore\cite{scherpen}, $Z=G^{-1} \cO$ and $Z^2=\C \cO$, where $\cO := \int_0^{\infty} e^{A^\top \tau}C^\top \sigma \sigma C e^{A \tau} d \tau = \int_0^{\infty} e^{A^\top \tau}C^\top C e^{A \tau} d \tau$ is the \emph{observability Gramian}.} 
\end{remark}
Since the signature matrix modified Hankel operator $\sigma \cH$ is self-adjoint the eigenfunctions $\phi_i$ are \emph{orthogonal} with respect to the $L_2(0,\infty)$ inner product. This means that for $i \neq j$ eigenvectors $x_i$ of $Z= \C G$ satisfy
\bq
0 = \int_0^{\infty} x_i^\top \C^{-1} e^{At} B B^\top e^{A^\top t} \C^{-1}x_j dt = x_i^\top \C^{-1} x_j
\eq
Furthermore, the orthogonal eigenfunctions $\phi_1, \cdots,\phi_n$ can be turned into an \emph{orthonormal} set by scaling $x_i$ such that $x_i \C^{-1}x_i = 1, \quad i=1,\cdots,n$. Then one verifies that the resulting eigenfunctions $\phi_1, \cdots, \phi_n$ satisfy $B^\top e^{At} G e^{A\tau}B = \sum_{i=1}^n \lambda_i \phi_i(t) \phi_i^\top(\tau)$.
This can be considered as a (finite-dimensional) version of Mercer's theorem (where the eigenvalues $\lambda_1, \cdots, \lambda_n$ need not be nonnegative). Summarizing:
\begin{proposition}
Consider an $n$-dimensional system $\Sigma=(A,B,C,D)$ with Hankel operator $\cH$. Suppose the system is reciprocal with respect to $\sigma$. Then $\sigma \cH$ has, apart from zero eigenvalues, $n$ real eigenvalues $\lambda_1, \cdots,\lambda_n$, given as the eigenvalues of $\C G$, where $G=G^\top$ is determined by \eqref{G12}, and $\C$ is the controllability Gramian. Furthermore $\lambda_i \neq 0, i=1, \cdots,n,$ if and only if the system is controllable. The corresponding eigenfunctions are given as
\bq
\phi_i(t)= B^\top e^{A^\top t} \C^{-1} x_i, \quad i=1, \cdots, n,
\eq
where $x_i, i=1, \cdots,n,$ are eigenvectors of $\C G$. By scaling $x_i$ such that $x_i\C^{-1}x_i=1, i=1, \cdots,n,$ it follows that $\phi_1, \cdots,\phi_n$ is an orthonormal set in $L_2[0,\infty)$, and the impulse response matrix $W(t-\tau)$ of the system satisfies
\bq
\sigma W(t+\tau)= B^\top e^{At} G e^{A\tau}B = \sum_{i=1}^n \lambda_i \phi_i(t) \phi_i^\top(\tau)
\eq
Furthermore $\lambda_i > 0, i=1, \cdots,n,$ if and only if $G> 0$.
\end{proposition}

\end{appendices}




\end{document}